\documentclass[reqno]{amsart}

\newtheorem{theorem}{Theorem}
\newtheorem{lemma}[theorem]{Lemma}
\newtheorem{proposition}[theorem]{Proposition}

\theoremstyle{definition}
\newtheorem{definition}[theorem]{Definition}
\usepackage{graphicx}
\usepackage[all]{xy}
\usepackage{pstricks, enumerate, pst-node, pst-text, pst-plot}

\theoremstyle{remark}
\newtheorem{remark}[theorem]{\bf Remark}

\numberwithin{equation}{section}
\numberwithin{theorem}{section}

\newcommand{\intav}[1]{\mathchoice {\mathop{\vrule width 6pt height 3 pt depth  -2.5pt
\kern -8pt \intop}\nolimits_{\kern -6pt#1}} {\mathop{\vrule width
5pt height 3  pt depth -2.6pt \kern -6pt \intop}\nolimits_{#1}}
{\mathop{\vrule width 5pt height 3 pt depth -2.6pt \kern -6pt
\intop}\nolimits_{#1}} {\mathop{\vrule width 5pt height 3 pt depth
-2.6pt \kern -6pt \intop}\nolimits_{#1}}}

\newcommand{\intavl}[1]{\mathchoice {\mathop{\vrule width 6pt height 3 pt depth  -2.5pt
\kern -8pt \intop}\limits_{\kern -6pt#1}} {\mathop{\vrule width 5pt
height 3  pt depth -2.6pt \kern -6pt \intop}\nolimits_{#1}}
{\mathop{\vrule width 5pt height 3 pt depth -2.6pt \kern -6pt
\intop}\nolimits_{#1}} {\mathop{\vrule width 5pt height 3 pt depth
-2.6pt \kern -6pt \intop}\nolimits_{#1}}}

\newcommand{\R}{\mathbb{R}}
\newcommand{\N}{\mathbb{N}}

\newcommand{\al}{\alpha}

\newcommand{\la}{\lambda}
\newcommand{\te}{{\theta}}
\newcommand{\Te}{{\Theta}}

\newcommand{\ve}{{\varepsilon}}
\newcommand{\vr}{{\varphi}}

\newcommand{\und}{\underline}
\def\ovt{\underline{\theta}}
\def\proj{{\rm proj}}
\def\hd{{\rm HD}}

\begin{document}

\title[Marstrand's Theorem for products of Cantor sets]{A combinatorial proof of Marstrand's Theorem for products of regular Cantor sets}

\author[Yuri Lima]{Yuri Lima}
\address{Instituto Nacional de Matem\'atica Pura e Aplicada, Estrada Dona Castorina 110, 22460-320, Rio de Janeiro, Brasil.}
\email{yurilima@impa.br}

\author[Carlos Gustavo Moreira]{Carlos Gustavo Moreira}
\address{Instituto Nacional de Matem\'atica Pura e Aplicada, Estrada Dona Castorina 110, 22460-320, Rio de Janeiro, Brasil.}
\email{gugu@impa.br}
\




\begin{abstract}
In a paper from 1954 Marstrand proved that if $K\subset\R^2$ has Hausdorff dimension greater than $1$, then
its one-dimensional projection has positive Lebesgue measure for almost-all directions. In this article, we
give a combinatorial proof of this theorem when $K$ is the product of regular Cantor sets of class $C^{1+\al}$,
$\al>0$, for which the sum of their Hausdorff dimension is greater than $1$.
\end{abstract}

\maketitle

\section{Introduction}

If $U$ is a subset of $\R^n$, the diameter of $U$ is $|U|=\sup\{|x-y|;x,y\in U\}$
and, if $\mathcal U$ is a family of subsets of $\R^n$, the {\it diameter} of $\mathcal U$ is defined as
$$\left\|\mathcal U\right\|=\sup_{U\in\,\mathcal U}|U|.$$
Given $d>0$, the {\it Hausdorff $d$-measure} of a set $K\subseteq\R^n$ is
$$m_d(K)=\lim_{\ve\rightarrow 0}\left(\inf_{\mathcal U\text{ covers }K\atop{\left\|\mathcal U\right\|<\,\ve}}
\sum_{U\in\,\mathcal U}|U|^d\right).$$
In particular, when $n=1$, $m=m_1$ is the Lebesgue measure of Lebesgue measurable sets on $\R$.
It is not difficult to show that there exists a unique $d_0\ge0$ for which $m_d(K)=+\infty$ if $d<d_0$ and
$m_d(K)=0$ if $d>d_0$. We define the Hausdorff dimension of $K$ as ${\rm HD}(K)=d_0$. Also, for each $\te\in\R$,
let $v_\te=(\cos \te,\sin\te)$, $L_\te$ the line in $\R^2$ through the origin containing $v_\te$ and
$\proj_\te:\R^2\rightarrow L_\te$ the orthogonal projection. From now on, we'll restrict $\te$ to
the interval $[-\pi/2,\pi/2]$, because $L_{\te}=L_{\te+\pi}$.

In $1954$, J. M. Marstrand \cite{Ma} proved the following result on the fractal dimension of plane sets.
\vspace{.2cm}

\noindent{\bf Theorem. }{\it If $K\subseteq\R^2$ is a Borel set such that $\hd(K)>1$, then
$m(\proj_\te(K))>0$ for $m$-almost every $\te\in\R$.}
\vspace{.2cm}

The proof is based on a qualitative characterization of the ``bad" angles
$\te$ for which the result is not true. Specifically, Marstrand exhibits a Borel measurable function
$f(x,\te)$, $(x,\te)\in\R^2\times[-\pi/2,\pi/2]$, such that $f(x,\te)=+\infty$ for $m_d$-almost every $x\in K$,
for every ``bad" angle. In particular,
\begin{equation}\label{equacao 1}
\int_K f(x,\te)dm_d(x)=+\infty.
\end{equation}
On the other hand, using a version of Fubini's Theorem, he proves that
$$\int_{-\pi/2}^{\pi/2}d\te\int_K f(x,\te)dm_d(x)=0$$
which, in view of (\ref{equacao 1}), implies that
$$m\left(\{\te\in[-\pi/2,\pi/2]\,;\, m(\proj_\te(K))=0\}\right)=0.$$
These results are based on the analysis of rectangular densities
of points.

Many generalizations and simpler proofs have appeared since. One of them came in 1968 by R. Kaufman
who gave a very short proof of Marstrand's theorem using methods of potential theory. See \cite{K}
for his original proof and \cite{M1}, \cite{PT} for further discussion.

In this article, we prove a particular case of Marstrand's Theorem.

\begin{theorem}\label{theorem 1}
If $K_1,K_2$ are regular Cantor sets of class $C^{1+\al}$, $\al>0$, such that
$d={\rm HD}(K_1)+{\rm HD}(K_2)>1$, then $m\left(\proj_\te(K_1\times K_2)\right)>0$ for $m$-almost every $\te\in\R$.
\end{theorem}

The argument also works to show that the push-forward measure of the restriction of $m_d$ to $K_1\times K_2$,
defined as $\mu_\te=(\proj_\te)_*(m_d|_{K_1\times K_2})$, is absolutely continuous with respect to $m$,
for $m$-almost every $\te\in\R$.
Denoting its Radon-Nykodim derivative by $\chi_\te=d\mu_\te/dm$, we also prove the following result.

\begin{theorem}\label{theorem 2}
$\chi_\te$ is an $L^2$ function for $m$-almost every $\te\in\R$.
\end{theorem}

\begin{remark}
Theorem \ref{theorem 2}, as in this work, follows from most proofs of Marstrand's theorem
and, in particular, is not new as well.
\end{remark}

Our proof makes a study on the fibers
${\proj_\te}^{-1}(v)\cap(K_1\times K_2)$, $(\te,v)\in\R\times L_\te$, and relies on two facts:

\noindent (I) A regular Cantor set of Hausdorff dimension $d$ is regular in the sense that the $m_d$-measure
of small portions of it has the same exponential behavior.

\noindent (II) This enables us to conclude that, except for a small set of angles $\te\in\R$, the fibers
${\proj_\te}^{-1}(v)\cap(K_1\times K_2)$ are not concentrated in a thin region. As a consequence,
$K_1\times K_2$ projects into a set of positive Lebesgue measure.

The idea of (II) is based on the work \cite{Mo} of the second author. He proves that, if $K_1$ and $K_2$ are
regular Cantor sets of class $C^{1+\al}$, $\al>0$, and at least one of them is non-essentially affine
(a technical condition), then
the arithmetic sum $K_1+K_2=\{x_1+x_2;x_1\in K_1,x_2\in K_2\}$ has the expected Hausdorff dimension:
$${\rm HD}(K_1+K_2)=\min\{1,{\rm HD}(K_1)+{\rm HD}(K_2)\}.$$

Marstrand's Theorem for products of Cantor sets has many useful applications in dynamical
systems. It is fundamental in certain results
of dynamical bifurcations, namely homoclinic bifurcations in surfaces. For instance, in \cite{PY}
it is used to show that hyperbolicity is not prevalent in homoclinic bifurcations associated to
horseshoes with Hausdorff dimension larger than one; in \cite{MY} it is used to prove that stable
intersections of regular Cantor sets are dense in the region where the sum of their Hausdorff dimensions
is larger than one; in \cite{MY2} to show that, for homoclinic bifurcations associated to horseshoes with
Hausdorff dimension larger than one, typically there are open sets of parameters with positive Lebesgue
density at the initial bifurcation parameter corresponding to persistent homoclinic tangencies.

\section{Regular Cantor sets of class $C^{1+\al}$}

We say that $K\subset\R$ is a {\it regular Cantor set of class $C^{1+\al}$}, $\al>0$, if:
\begin{itemize}
\item[(i)] there are disjoint compact intervals $I_1,I_2,\dots,I_r\subseteq[0,1]$ such that $K\subset I_1\cup\cdots\cup I_r$
and the boundary of each $I_i$ is contained in $K$;
\item[(ii)] there is a $C^{1+\al}$ expanding map $\psi$ defined in a neighbourhood of $I_1\cup I_2\cup\cdots\cup I_r$
such that $\psi(I_i)$ is the convex hull of a finite union of some intervals $I_j$, satisfying:
\begin{itemize}
\item[(ii.1)] for each $i\in\{1,2,\ldots,r\}$ and $n$ sufficiently big, $\psi^n(K\cap I_i)=K$;
\item[(ii.2)] $K=\bigcap\limits_{n\in\N}\psi^{-n}(I_1\cup I_2\cup\cdots\cup I_r)$.
\end{itemize}
\end{itemize}

The set $\{I_1,\ldots,I_r\}$ is called a Markov partition of $K$. It defines an $r\times r$ matrix $B=(b_{ij})$ by
\begin{eqnarray*}
b_{ij}&=&1,\ \ \text{ if } \psi(I_i)\supseteq I_j\\
      &=&0,\ \ \text{ if } \psi(I_i)\cap I_j=\emptyset,
\end{eqnarray*}
which encodes the combinatorial properties of $K$. Given such matrix, consider the set
$\Sigma_B=\left\{\ovt=(\te_1,\te_2,\ldots)\in\{1,\ldots,r\}^\N\,;\,b_{\te_i\te_{i+1}}=1,\forall\, i\ge1\right\}$
and the shift transformation $\sigma:\Sigma_B\rightarrow\Sigma_B$ given by
$\sigma(\te_1,\te_2,\ldots)=(\te_2,\te_3,\ldots)$.

There is a natural homeomorphism between the pairs $(K,\psi)$ and $(\Sigma_B,\sigma)$. For each finite word
$\und a=(a_1,\ldots,a_n)$ such that $b_{a_ia_{i+1}}=1$, $i=1,\ldots,n-1$, the intersection
$$I_{\und a}=I_{a_1}\cap \psi^{-1}(I_{a_2})\cap\cdots\cap\psi^{-(n-1)}(I_{a_{n}})$$
is a non-empty interval with diameter $|I_{\und a}|=|I_{a_n}|/|(\psi^{n-1})'(x)|$ for some $x\in I_{\und a}$,
which is exponentially small if $n$ is large. Then, $\{h(\ovt)\}=\bigcap_{n\ge 1}I_{(\te_1,\ldots,\te_n)}$
defines a homeomorphism  $h:\Sigma_B\rightarrow K$ that commutes the diagram
$$
\xymatrix{
\Sigma_B\ar[r]^{\sigma} \ar[d]_{h}&\Sigma_B\ar[d]^{h}\\
K\ar[r]_{\psi}&K}
$$
If $\lambda=\sup\{|\psi'(x)|;x\in I_1\cup\cdots\cup I_r\}\in(1,+\infty)$, then
$\left|I_{(\te_1,\ldots,\te_{n+1})}\right|\ge\lambda^{-1}\cdot \left|I_{(\te_1,\ldots,\te_n)}\right|$
and so, for $\rho>0$ small and $\ovt\in\Sigma_B$, there is a positive integer $n=n(\rho,\ovt)$ such that
$$\rho\le\left|I_{(\te_1,\ldots,\te_n)}\right|\le\lambda\rho.$$

\begin{definition}
A $\rho$-decomposition of $K$ is any finite set $(K)_\rho=\{I_1,I_2,\ldots,I_r\}$ of disjoint closed
intervals of $\R$, each one of them intersecting $K$, whose union covers $K$ and such that
$$\rho\le |I_i|\le \lambda\rho\,,\ i=1,2,\ldots,r.$$
\end{definition}

\begin{remark}
Although $\rho$-decompositions are not unique, we use, for simplicity, the notation $(K)_\rho$ to denote
any of them. We also use the same notation $(K)_\rho$ to denote the
set $\cup_{I\in(K)_\rho}I\subset\R$ and the distinction between these two situations will be clear
throughout the text.
\end{remark}

Every regular Cantor set of class $C^{1+\al}$ has a $\rho$-decomposition for $\rho>0$ small: by the compactness
of $K$, the family $\left\{I_{\left(\te_1,\ldots,\te_{n(\rho,\ovt)}\right)}\right\}_{\ovt\in\Sigma_B}$ has a finite
cover (in fact, it is only necessary for $\psi$ to be of class $C^1$). Also, one can define $\rho$-decomposition
for the product of two Cantor sets $K_1$ and $K_2$, denoted by $(K_1\times K_2)_\rho$. Given $\rho\not=\rho'$
and two decompositions
$(K_1\times K_2)_{\rho'}$ and $(K_1\times K_2)_\rho$, consider the partial order
$$(K_1\times K_2)_{\rho'}\prec(K_1\times K_2)_\rho\ \iff\ \rho'<\rho\text{ and }
\bigcup_{Q'\in(K_1\times K_2)_{\rho'}}Q'\subseteq\bigcup_{Q\in(K_1\times K_2)_{\rho}}Q.$$
In this case, $\proj_\te((K_1\times K_2)_{\rho'})\subseteq\proj_\te((K_1\times K_2)_{\rho})$ for any $\te$.

A remarkable property of regular Cantor sets of class $C^{1+\al}$, $\al>0$, is bounded distortion.

\begin{lemma}\label{lemma 1}
Let $(K,\psi)$ be a regular Cantor set of class $C^{1+\al}$, $\al>0$, and $\{I_1,\ldots,I_r\}$
a Markov partition. Given $\delta>0$, there exists a constant $C(\delta)>0$, decreasing on $\delta$, with
the following property: if $x,y\in K$ satisfy
\begin{itemize}
\item[(i)] $|\psi^n(x)-\psi^n(y)|<\delta$;
\item[(ii)] The interval $[\psi^i(x),\psi^i(y)]$ is contained in $I_1\cup\cdots\cup I_r$, for $i=0,\ldots,n-1$,
\end{itemize}
then
$$e^{-C(\delta)}\le\dfrac{\left|(\psi^n)'(x)\right|}{\left|(\psi^n)'(y)\right|}\le e^{C(\delta)}\,.$$
In addition, $C(\delta)\rightarrow 0$ as $\delta\rightarrow 0$.
\end{lemma}

A direct consequence of bounded distortion is the required regularity of $K$, contained in the next result.

\begin{lemma}\label{lemma 2} Let $K$ be a regular Cantor set of class $C^{1+\al}$, $\al>0$, and let $d={\rm HD}(K)$. Then
$0<m_d(K)<+\infty$. Moreover, there is $c>0$ such that, for any $x\in K$ and $0\le r\le 1$,
$$c^{-1}\cdot r^d\le m_d(K\cap B_r(x))\le c\cdot r^d.$$
\end{lemma}

The same happens for products $K_1\times K_2$ of Cantor sets (without loss of generality, considered with the
box norm).

\begin{lemma}\label{lemma 3}
Let $K_1,K_2$ be regular Cantor sets of class $C^{1+\al}$, $\al>0$, and let
$d={\rm HD}(K_1)+{\rm HD}(K_2)$. Then $0<m_d(K_1\times K_2)<+\infty$. Moreover, there is $c_1>0$
such that, for any $x\in K_1\times K_2$ and $0\le r\le 1$,
$${c_1}^{-1}\cdot r^d\le m_d\left((K_1\times K_2)\cap B_r(x)\right)\le c_1\cdot r^d.$$
\end{lemma}

See chapter 4 of \cite{PT} for the proofs of these lemmas. In particular, if $Q\in(K_1\times K_2)_\rho$, there is
$x\in (K_1\cup K_2)\cap Q$ such that $B_{\lambda^{-1}\rho}(x)\subseteq Q\subseteq B_{\lambda\rho}(x)$ and so
$$\left({c_1\lambda^d}\right)^{-1}\cdot\rho^d\le m_d((K_1\times K_2)\cap Q)\le c_1\lambda^d\cdot\rho^d.$$
Changing $c_1$ by $c_1\lambda^d$, we may also assume that
$${c_1}^{-1}\cdot \rho^d\le m_d\left((K_1\times K_2)\cap Q\right)\le c_1\cdot \rho^d,$$
which allows us to obtain estimates on the cardinality of $\rho$-decompositions.

\begin{lemma}\label{lemma 4}
Let $K_1,K_2$ be regular Cantor sets of class $C^{1+\al}$, $\al>0$, and let
$d={\rm HD}(K_1)+{\rm HD}(K_2)$. Then there is $c_2>0$ such that, for any $\rho$-decomposition
$(K_1\times K_2)_\rho$, $x\in K_1\times K_2$ and $0\le r\le 1$,
$$\#\left\{Q\in(K_1\times K_2)_\rho; Q\subseteq B_r(x)\right\}\le c_2\cdot \left(\dfrac{r}{\rho}\right)^d\cdot$$
In addition, ${c_2}^{-1}\cdot \rho^{-d}\le\#(K_1\times K_2)_\rho\le c_2\cdot \rho^{-d}$.
\end{lemma}

\begin{proof}
We have
\begin{eqnarray*}
c_1\cdot r^d&\ge&m_d\left((K_1\times K_2)\cap B_r(x)\right)\\
            &\ge&\sum_{Q\subseteq B_r(x)}m_d\left((K_1\times K_2)\cap Q\right)\\
            &\ge&\sum_{Q\subseteq B_r(x)}{c_1}^{-1}\cdot\rho^d\\
            &=&\#\left\{Q\in(K_1\times K_2)_\rho; Q\subseteq B_r(x)\right\}\cdot {c_1}^{-1}\cdot \rho^d
\end{eqnarray*}
and then
$$\#\left\{Q\in(K_1\times K_2)_\rho; Q\subseteq B_r(x)\right\}\le {c_1}^2\cdot\left(\dfrac{r}{\rho}\right)^d\cdot$$
On the other hand,
$$
m_d(K_1\times K_2)=\sum_{Q\in(K_1\times K_2)_\rho}m_d\left((K_1\times K_2)\cap Q\right)
\le\sum_{Q\in(K_1\times K_2)_\rho}c_1\cdot\rho^d,
$$
implying that
$$\#(K_1\times K_2)_\rho\ge {c_1}^{-1}\cdot m_d(K_1\times K_2)\cdot\rho^{-d}.$$
Taking $c_2=\max\{{c_1}^2\,,\,c_1/m_d(K_1\times K_2)\}$, we conclude the proof.
\end{proof}

\section{Proof of Theorem \ref{theorem 1}}\label{proof of theorem 1}

Given rectangles $Q$ and $\tilde Q$, let
$$\Te_{Q,\tilde Q}=\left\{\te\in[-\pi/2,\pi/2];\proj_\te(Q)\cap\proj_\te(\tilde Q)\not=\emptyset\right\}.$$

\begin{lemma}\label{lemma 5}
If $Q,\tilde Q\in(K_1\times K_2)_\rho$ and
$x\in(K_1\times K_2)\cap Q,\tilde x\in(K_1\times K_2)\cap\tilde Q$, then
$$m\left(\Te_{Q,\tilde Q}\right)\le 2\pi\lambda\cdot\dfrac{\rho}{d(x,\tilde x)}\,\cdot$$
\end{lemma}

\begin{proof} Consider the figure.\\

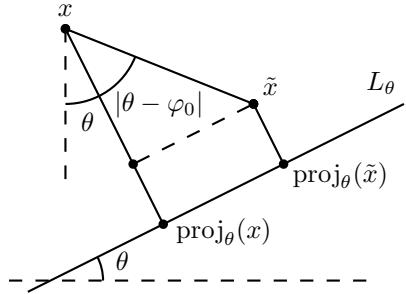
\begin{center}
\psset{unit=.5cm} \begin{pspicture}(-3,-.7)(8,6)

\psline(-2,-1)(8,4)\psline{*-*}(-1,6)(1.6,.8)\psline{*-*}(4,4)(4.8,2.4)\psline{*-*}(-1,6)(4,4)
\psline[linestyle=dashed,dash=5pt 5pt](-2.5,-.7)(7,-.7)\psline[linestyle=dashed,dash=5pt 5pt](-1,2)(-1,6)
\psline[linestyle=dashed,dash=5pt 5pt]{*-*}(.8,2.4)(4,4)\psarc(-1,6){2}{270}{338}\psarc(-1.4,-.7){1.4}{0}{27}

\uput[90](-1,6){$x$}\uput[45](4,4){$\tilde x$}\uput[-45](1.7,1.2){$\proj_\te(x)$}
\uput[-45](4.8,2.7){$\proj_\te(\tilde x)$}\uput[90](7.4,3.9){$L_\te$}\uput[-60](-.7,4.1){$\te$}
\uput[-60](.6,4.7){$|\te-\vr_0|$}\uput[30](0,-.5){$\te$}
\end{pspicture} \end{center}
Since $\proj_\te(Q)$ has diameter at most $\lambda\rho$,
$d(\proj_\te(x),\proj_\te(\tilde x))\le 2\lambda\rho$ and then, by elementary geometry,
\begin{eqnarray*}
\sin(|\te-\vr_0|)&=  &\dfrac{d(\proj_\te(x),\proj_\te(\tilde x))}{d(x,\tilde x)}\\
                 &\le&2\lambda\cdot\dfrac{\rho}{d(x,\tilde x)}\\
\Longrightarrow\hspace{1cm}|\te-\vr_0|&\le&\pi\lambda\cdot\dfrac{\rho}{d(x,\tilde x)}\,,
\end{eqnarray*}
because $\sin^{-1}y\le \pi y/2$. As $\vr_0$ is fixed, the lemma is proved.
\end{proof}

We point out that, although ingenuous, Lemma \ref{lemma 5} expresses the crucial property of transversality that
makes the proof work, and all results related to Marstrand's theorem use a similar idea in one way or another.
See \cite{R} where this tranversality condition is also exploited.

Fixed a $\rho$-decomposition $(K_1\times K_2)_\rho$, let
$$N_{(K_1\times K_2)_\rho}(\te)=\#\left\{(Q,\tilde Q)\in(K_1\times K_2)_\rho\times (K_1\times K_2)_\rho;
\proj_\te(Q)\cap\proj_\te(\tilde Q)\not=\emptyset\right\}$$
for each $\te\in[-\pi/2,\pi/2]$ and
$$E((K_1\times K_2)_\rho)=\int_{-\pi/2}^{\pi/2}N_{(K_1\times K_2)_\rho}(\te)d\te.$$

\begin{proposition}\label{proposition 1}
Let $K_1,K_2$ be regular Cantor sets of class $C^{1+\al}$, $\al>0$, and let
$d={\rm HD}(K_1)+{\rm HD}(K_2)$. Then there is $c_3>0$ such that, for any $\rho$-decomposition
$(K_1\times K_2)_\rho$,
$$E((K_1\times K_2)_\rho)\le c_3\cdot \rho^{1-2d}.$$
\end{proposition}

\begin{proof}
Let $s_0=\left\lceil\log_2 \rho^{-1}\right\rceil$ and choose, for each $Q\in(K_1\times K_2)_\rho$,
a point $x\in (K_1\times K_2)\cap Q$. By a double counting and using Lemmas \ref{lemma 4} and \ref{lemma 5}, we have
\begin{eqnarray*}
E((K_1\times K_2)_\rho)&=&\sum_{Q,\tilde Q\in (K_1\times K_2)_\rho}m\left(\Te_{Q,\tilde Q}\right)\\
&=&\sum_{s=1}^{s_0}\sum_{Q,\tilde Q\in (K_1\times K_2)_\rho\atop{2^{-s}< d(x,\tilde x)\le 2^{-s+1}}}
m\left(\Te_{Q,\tilde Q}\right)\\
&\le&\sum_{s=1}^{s_0}c_2\cdot\rho^{-d}\left[c_2\cdot\left(\dfrac{2^{-s+1}}{\rho}\right)^d\right]
\cdot\left(2\pi\lambda\cdot\dfrac{\rho}{2^{-s}}\right)\\
&=&2^{d+1}\pi\lambda{c_2}^2\cdot\left(\sum_{s=1}^{s_0}2^{s(1-d)} \right)\cdot\rho^{1-2d}.
\end{eqnarray*}
Because $d>1$, $c_3=2^{d+1}\pi\lambda{c_2}^2\cdot\sum_{s\ge1}2^{s(1-d)}<+\infty$ satisfies the
required inequality.
\end{proof}

This implies that, for each $\ve>0$, the upper bound
\begin{equation}\label{equacao 2}
N_{(K_1\times K_2)_\rho}(\te)\le\dfrac{c_3\cdot\rho^{1-2d}}{\ve}
\end{equation}
holds for every $\te$ except for a set of measure at most $\ve$. Letting
$c_4={c_2}^{-2}\cdot{c_3}^{-1}$, we will show that
\begin{equation}\label{equacao 3}
m\left(\proj_\te\left((K_1\times K_2)_\rho\right)\right)\ge c_4\cdot\ve
\end{equation}
for every $\te$ satisfying (\ref{equacao 2}). For this, divide $[-2,2]\subseteq L_\te$ in
$\left\lfloor 4/\rho\right\rfloor$ intervals $J_1^\rho,\ldots,$ $J_{\left\lfloor 4/\rho\right\rfloor}^\rho$
of equal lenght (at least $\rho$) and define
$$s_{\rho,i}=\#\left\{Q\in(K_1\times K_2)_\rho\,;\,\proj_\te(x)\in J_i^\rho\right\},\ \ i=1,\ldots,\left\lfloor 4/\rho\right\rfloor.$$
Then $\sum_{i=1}^{\left\lfloor 4/\rho\right\rfloor}s_{\rho,i}=\#(K_1\times K_2)_\rho$ and
$$\sum_{i=1}^{\left\lfloor 4/\rho\right\rfloor}{s_{\rho,i}}^2\le N_{(K_1\times K_2)_\rho}(\te)
\le c_3\cdot\rho^{1-2d}\cdot\ve^{-1}.$$
Let $S_\rho=\{1\le i\le\left\lfloor 4/\rho\right\rfloor; s_{\rho,i}>0\}$. By Cauchy-Schwarz inequality,
$$\#S_\rho\ge\dfrac{\displaystyle\left(\sum_{i\in S_\rho}s_{\rho,i}\right)^2}{\displaystyle\sum_{i\in S_\rho}{s_{\rho,i}}^2}
\ge\dfrac{{c_2}^{-2}\cdot\rho^{-2d}}{c_3\cdot\rho^{1-2d}\cdot\ve^{-1}}=\dfrac{c_4\cdot\ve}{\rho}\,\cdot$$
For each $i\in S_\rho$, the interval $J_i^\rho$ is contained in $\proj_\te((K_1\times K_2)_\rho)$ and then
$$m\left(\proj_\te((K_1\times K_2)_\rho)\right)\ge c_4\cdot\ve,$$
which proves (\ref{equacao 3}).

\begin{proof}[Proof of Theorem \ref{theorem 1}]
Fix a decreasing sequence
\begin{equation}\label{equacao 6}
(K_1\times K_2)_{\rho_1}\succ(K_1\times K_2)_{\rho_2}\succ\cdots
\end{equation}
of decompositions such that $\rho_n\rightarrow0$ and, for each $\ve>0$, consider the sets
$$G_\ve^n=\left\{\te\in[-\pi/2,\pi/2]\,;\,N_{(K_1\times K_2)_{\rho_n}}(\te)
\le c_3\cdot{\rho_n}^{1-2d}\cdot\ve^{-1}\right\},\ \ n\ge 1.$$
Then $m\left([-\pi/2,\pi/2]\backslash G_\ve^n\right)\le\ve$,
and the same holds for the set
$$G_\ve=\bigcap_{n\ge 1}\bigcup_{l=n}^{\infty}G_\ve^l\,.$$
If $\te\in G_\ve$, then
$$m\left(\proj_\te((K_1\times K_2)_{\rho_n})\right)\ge c_4\cdot\ve\,,\ \text{ for infinitely many }n,$$
which implies that $m\left(\proj_\te(K_1\times K_2)\right)\ge c_4\cdot\ve$.
Finally, the set $G=\cup_{n\ge 1}G_{1/n}$ satisfies $m([-\pi/2,\pi/2]\backslash G)=0$ and
$m\left(\proj_\te(K_1\times K_2)\right)>0$, for any $\te\in G$.
\end{proof}

\section{Proof of Theorem \ref{theorem 2}}

Given any $X\subset K_1\times K_2$, let $(X)_\rho$ be the restriction of the $\rho$-decomposition
$(K_1\times K_2)_\rho$ to those rectangles
which intersect $X$. As done in Section \ref{proof of theorem 1}, we'll obtain estimates on the
cardinality of $(X)_\rho$. Being a subset of $K_1\times K_2$, the upper estimates from Lemma \ref{lemma 4}
also hold for $X$. The lower estimate is given by

\begin{lemma}\label{lemma 6}
Let $X$ be a subset of $K_1\times K_2$ such that $m_d(X)>0$. Then there is $c_6=c_6(X)>0$ such
that, for any $\rho$-decomposition $(K_1\times K_2)_\rho$ and $0\le r\le 1$,
$$c_6\cdot\rho^{-d}\le \#(X)_\rho\le c_2\cdot\rho^{-d}.$$
\end{lemma}

\begin{proof}
As $m_d(X)<+\infty$, there exists $c_5=c_5(X)>0$ (see Theorem 5.6 of \cite{F}) such that
$$m_d\left(X\cap B_r(x)\right)\le c_5\cdot r^d\,,\ \text{ for all }x\in X\text{ and }0\le r\le 1,$$
and then
$$m_d(X)=\sum_{Q\in (X)_\rho}m_d\left(X\cap Q\right)\le\sum_{Q\in (X)_\rho}c_5\cdot(\la\rho)^d
=\left({c_5}\cdot\la^d\right)\cdot\rho^d\cdot\#(X)_\rho\,.$$
Just take $c_6={c_5}^{-1}\cdot\la^{-d}\cdot m_d(X)$.
\end{proof}

\begin{proposition}\label{proposition 2}
The measure $\mu_\te=(\proj_\te)_*(m_d|_{K_1\times K_2})$ is absolutely continuous with respect to $m$, for
$m$-almost every $\te\in\R$.
\end{proposition}

\begin{proof}
Note that the implication
\begin{equation}\label{equacao 4}
X\subset K_1\times K_2\, ,\ m_d(X)>0\ \ \Longrightarrow\ \ m(\proj_\te(X))>0
\end{equation}
is sufficient for the required absolute continuity. In fact, if $Y\subset L_\te$ satisfies $m(Y)=0$, then
$$\mu_\te(Y)=m_d(X)=0\,,$$
where $X={\proj_\te}^{-1}(Y)$. Otherwise, by (\ref{equacao 4}) we would have $m(Y)=m(\proj_\te(X))>0$,
contradicting the assumption.

We prove that (\ref{equacao 4}) holds for every $\te\in G$, where $G$ is the set defined in the proof of Theorem \ref{theorem 1}.
The argument is the same made after Proposition \ref{proposition 1}: as, by the previous lemma,
$\#(X)_\rho$ has lower and upper estimates depending only on $X$ and $\rho$, we obtain that
$$m\left(\proj_\te((X)_{\rho_n})\right)\ge {c_3}^{-1}\cdot{c_6}^2\cdot\ve\,,\ \text{ for infinitely many }n,$$
and then $m(\proj_\te(X))>0$.
\end{proof}

Let $\chi_\te=d\mu_\te/dm$. In principle, this is a $L^1$ function. We prove that it is a $L^2$ function,
for every $\te\in G$.

\begin{proof}[Proof of Theorem 2.]
Let $\te\in G_{1/m}$, for some $m\in\N$. Then
\begin{equation}\label{equacao 5}
N_{(K_1\times K_2)_{\rho_n}}(\te)\le c_3\cdot {\rho_n}^{1-2d}\cdot m\,,\ \text{for infinitely many }n.
\end{equation}
For each of these $n$, consider the partition $\mathcal P_n=\{J_1^{\rho_n},\ldots,J_{\lfloor 4/\rho_n\rfloor}^{\rho_n}\}$
of $[-2,2]\subset L_\te$ into intervals of equal length and let
$\chi_{\te,n}$ be the expectation of $\chi_\te$ with respect to $\mathcal P_n$. As $\rho_n\rightarrow 0$,
the sequence of functions $(\chi_{\te,n})_{n\in\N}$ converges pointwise to $\chi_\te$. By Fatou's Lemma,
we're done if we prove that each $\chi_{\te,n}$ is $L^2$ and its $L^2$-norm
$\left\|\chi_{\te,n}\right\|_2$ is bounded above by a constant independent of $n$.

By definition,
$$\mu_\te(J_i^{\rho_n})=m_d\left(({\proj_\te})^{-1}\left(J_i^{\rho_n}\right)\right)\le s_{\rho_n,i}
\cdot c_1\cdot{\rho_n}^d,\ \ i=1,2,\ldots,{\lfloor 4/\rho_n\rfloor},$$
and then
$$\chi_{\te,n}(x)=\dfrac{\mu_\te(J_i^{\rho_n})}{\left|J_i^{\rho_n}\right|}\le\dfrac{c_1\cdot s_{\rho_n,i} \cdot{\rho_n}^d}{\left|J_i^{\rho_n}\right|}\ ,\ \ \forall\,x\in J_i^{\rho_n},$$
implying that
\begin{eqnarray*}
\left\|\chi_{\te,n}\right\|_2^2&=&\int_{L_\te}\left|\chi_{\te,n}\right|^2dm\\
&=&\sum_{i=1}^{\lfloor 4/\rho_n\rfloor}\int_{J_i^{\rho_n}}\left|\chi_{\te,n}\right|^2dm\\
&\le&\sum_{i=1}^{\lfloor 4/\rho_n\rfloor}|J_i^{\rho_n}|\cdot\left(\dfrac{c_1\cdot s_{\rho_n,i}\cdot{\rho_n}^d}{|J_i^{\rho_n}|}\right)^2\\
&\le&{c_1}^2\cdot {\rho_n}^{2d-1}\cdot\sum_{i=1}^{\lfloor 4/\rho_n\rfloor}{s_{\rho_n,i}}^2\\
&\le&{c_1}^2\cdot {\rho_n}^{2d-1}\cdot N_{(K_1\times K_2)_{\rho_n}}(\te).
\end{eqnarray*}
In view of (\ref{equacao 5}), this last expression is bounded above by
$$\left({c_1}^2\cdot {\rho_n}^{2d-1}\right)\cdot\left(c_3\cdot {\rho_n}^{1-2d}\cdot m\right)={c_1}^2\cdot c_3\cdot m\,,$$
which is a constant independent of $n$.
\end{proof}

\section{Concluding remarks}

The proofs of Theorems \ref{theorem 1} and \ref{theorem 2} work not just for the case of products of
regular Cantor sets, but in greater generality, whenever $K\subset\R^2$ is a Borel set for which
there is a constant $c>0$ such that, for any $x\in K$ and $0\le r\le 1$,
$$c^{-1}\cdot r^d\le m_d(K\cap B_r(x))\le c\cdot r^d,$$
since this alone implies the existence of $\rho$-decompositions for $K$.

The good feature of the proof is that the discretization idea may be applied to other contexts. For example,
we prove in \cite{LM} a Marstrand type theorem in an arithmetical context.

\section*{Acknowledgments}
The authors are thankful to IMPA for the excellent ambient during the preparation of this
manuscript. The authors are also grateful to Carlos Matheus for carefully reading the preliminary version of this
work and the anonymous referee for many useful and detailed recommendations. This work was financially supported
by CNPq-Brazil and Faperj-Brazil.

\bibliographystyle{amsplain}

\end{document}